\documentclass[11pt]{article}
\usepackage[cp1251]{inputenc}
\usepackage{amsfonts}
\usepackage{amssymb}
\usepackage{amsthm}
\usepackage{amsmath}
\usepackage{xcolor}
\usepackage[english]{babel}
\newcommand{\be}{\begin{equation}}
\newcommand{\ee}{\end{equation}}

\newtheorem{thm}{Theorem}[section]

\newtheorem{defn}{Definition}[section]

\newtheorem{com}{Comment}[section]

\begin{document}

\title{On a class of critical Markov branching processes with
non-homogeneous Poisson immigration}
\author{Kosto V. Mitov, \\
Faculty of Pharmacy, Medical University, Pleven, Bulgaria \\
email: kmitov@yahoo.com \\
Nikolay M. Yanev, \\
Institute of Mathematics and Informatics, BAS, Sofia, Bulgaria, \\
email: yanev@math.bas.bg}
\maketitle

\begin{abstract}
The paper studies a class of critical Markov branching processes with
infinite variance of the offspring distribution. The processes admit also an
immigration component at the jump-points of a non-homogeneous Poisson
process, assuming that the mean number of immigrants is infinite and the
intensity of the Poisson process converges to a constant. The asymptotic
behavior of the probability for non-visiting zero is obtained. Proper limit
distributions are proved, under suitable normalization of the sample paths,
depending on the offspring distribution and the distribution of the
immigrants.
\end{abstract}

\textbf{Key words:} Infinite variance; Limit theorems; Markov branching
process; Non-homogeneous immigration

\textbf{2020 Mathematics Subject Classification:} Primary 60J80; Secondary
60F05, 60J85, 62P10

\section{Introduction}

The paper deals with Markov branching processes with immigration in the
time-moments generated by Poisson measure with a local intensity $r(t).$ We
consider the critical case when the offspring mean is equal to one, but the
offspring variance is infinite. The distribution of immigrants belongs to
the class of stable laws with infinite mean and $r(t)$ converges to some
positive constant.

Recall that the terminology \textit{branching process} was proposed by
Kolmogorov and appeared officially in \cite{Kolmogorov} where the multitype
Markov branching processes were introduced. Further developments are
presented in \cite{Harris}, \cite{sevast} and \cite{an}. The first branching
process with immigration was formulated by Sevastyanov \cite{Sevastyanov57}.
He investigated a single-type Markov process in which immigration occurs
according to a time-homogeneous Poisson process, and proved limiting
distributions. Branching processes with time-nonhomogeneous immigration were
first proposed by Durham \cite{Durham} and Foster and Williamson \cite%
{FosterWil}. Further results can be found in Badalbaev and Rahimov \cite%
{BadRah} and Rahimov \cite{Rahimov}. See also a review paper of Rahimov \cite%
{Rahimov2}. A model with nonhomogeneous migration was investigated by Yanev
and Mitov \cite{YanMit85}. Multitype Markov branching processes with
nonhomogeneous Poisson immigration were considered by Mitov et al. \cite{myh}
and Slavtchova-Bojkova et al. \cite{Mar22}, \cite{Mar23}. Notice that the
limiting distributions in these models were obtained in the case of finite
first and second offspring characteristics as well as those of the
immigration components.

Pakes \cite{pakes1975}, \cite{pakes2010} investigated respectively Bienaym%
\'{e}-Galton-Watson process and Markov branching process with infinite
offspring variance and finite mean of the immigrants. Imomov and Tukhtaev
\cite{Imomov} considered critical Bienaym\'{e}-Galton-Watson process with
infinite offspring variance and infinite mean of immigrants and extended
also some of the results of Pakes \cite{pakes1975}. Sagitov \cite{sagitov}
studied multi-type Markov branching processes in the case of homogeneous
Poisson immigration with infinite second moments of the offspring
distributions and infinite first moment of the number of immigrants.

Branching processes with time-nonhomogeneous immigration find applications
for investigating the dynamics of biological systems, particularly cellular
populations (see, for example, \cite{yay1,Hyr1,Hyr2}). In these
applications, the stem cells often are considered as an immigration
component.

Another interesting approach to branching processes where the immigration
component is generation-dependent is given in Gonzalez et al. \cite{Gonzalez}%
. Barczy et al. \cite{Barczy} investigated critical two-type decomposable
Bienaym\'{e}-Galton-Watson process with immigration.

We have to mention that some of the results obtained here are similar to
some of the results obtained in the discrete time case by Rahimov \cite%
{rah86,rah93mn} for Bienaym\'{e}-Galton-Watson branching processes and this
is not surprising. Let us note also that the methods of studying in the
present work are based on the functional equations for the probability
generating functions, stationary measures and some other methods which
essentially differ from the methods used in \cite{rah86,rah93mn}.

A detailed description of the considered models is presented in Section 2.
Some preliminary results and basic assumptions are given in Section 3. The
asymptotic behavior of the probabilities of non-visiting zero is
investigated in Section 4. Under the same basic conditions four types of
limiting distributions are obtained in Section 5. Surprisingly, the first
one (after a suitable normalization) is just a stable distribution with
parameter $\alpha $ from the distribution of immigrants. The second limiting
distribution belongs to a normal domain of attraction of a stable law with
parameter $\alpha =\gamma $, where $\gamma $ is an offspring parameter. In
the third case a discrete conditional limiting distribution with infinite
mean is obtained. Under the suitable normalization (with a slowly varying
function) it is shown that the fourth limiting distribution is just uniform
in the unit interval..

\section{Description of the models}

A single type Markov branching process can be described as follows. The
particles of a given type evolve independently of each-other, lives random
time $\tau $ with exponential distribution function $G(t)=\mathbf{P}\left\{
\tau \leq t\right\} =1-e^{-\mu t}$, $t\geq 0,\mu >0,$ and at the end of its
life the particle produces random number $\xi \geq 0$ of new particles of
the same type. The number of particles $Z(t),t\geq 0,$ form the stochastic
process, known as Markov branching process (see\cite{an}, \cite{Harris}, and
\cite{sevast}). We assume as usually that this evolution started at time $%
t=0 $ with one new particle. Denote by $h(s)=\mathbf{E}\left[ s^{\xi }\right]
$ the offspring probability generating function (p.g.f.) and
\begin{equation*}
F(t;s)=\mathbf{E}\left[ s^{Z(t)}|Z(0)=1\right] ,\ \ \ t \geq 0,\ \ s\in
\lbrack 0,1],
\end{equation*}
the p.g.f. of the process $Z(t)$, $t\geq 0$.

It is well known that (see e.g. \cite{an}, \cite{Harris})
\begin{eqnarray}
&& \frac{\partial F(t;s)}{\partial t}=\mu \lbrack h\left(F(t;s)\right)
-F(t;s)],  \label{kolmo0}
\end{eqnarray}
with initial condition $F(0;s)=s$. Under mild regularity conditions, it is
the only solution of this equation in the class of p.g.f.

Let us now suppose that along the Markov branching process $Z(t)$ there is a
sequence of random vectors $\left( S_{k},I_{k}\right) $, $k=0,1,2,\ldots $,
independent of $Z(t)$, where
\begin{equation*}
0=S_{0}<S_{1}<S_{2}<S_{3}<\cdots
\end{equation*}%
are the jump points of an non-homogeneous Poisson process $\nu (t)$
independent of $Z(t)$ and the random variables $\{I_{k}\}$\ are i.i.d. with
non-negative integer values. Denote by $r(t)$ the intensity of $\nu (t)$
with a mean measure $\displaystyle R(t)=\int_{0}^{t}r(u)du$. Let $%
\displaystyle g(s)=\mathbf{E}\left[s^{I_{k}}\right]$ be the p.g.f. of the immigrants.

Assume that at every jump-point $S_{k}$, a random number $I_{k}$ of new
particles immigrate into the process $Z(t)$ and they participate in the
evolution as the other particles. Let us denote the new process by $Y(t)$.
It can be strictly defined as follows
\begin{eqnarray*}
\displaystyle Y(t)=\sum_{k=1}^{\nu (t)}\sum_{j=1}^{I_{k}}Z^{(k,j)}
\left(t-S_{k}\right), \ \displaystyle t\geq 0,
\end{eqnarray*}
where $\left\{Z^{(k,j)}(t)\right\}$ are independent and identical copies of $%
Z(t)$.

\begin{defn}
The process $Y(t)$, $t\geq 0,$ is called Markov branching process with
non-homogeneous Poisson immigration (MBPNPI).
\end{defn}

The p.g.f. $\Phi(t;s):=\mathbf{E}\left[s^{Y(t)}\right]$ of the process $Y(t)$
has the following form
\begin{eqnarray}
&& \Phi(t;s)=\exp\left\{-\int_0^t r(t-u)(1 - g(F(u;s)))du \right\}, \
\Phi(t;0)=1.  \label{yanev}
\end{eqnarray}
The proof is given in \cite{yay1} and in the more general multitype case in
\cite{myh}.

For the intensity of the Poisson process, we assume additionally the
following condition
\begin{eqnarray}  \label{rt-3-g}
\displaystyle r(t) \to \rho>0, \ \ \displaystyle t \to \infty.
\end{eqnarray}

\section{Basic assumptions and preliminary results}

\label{sec2} For the branching mechanism we assume that the offspring p.g.f.
$f(s)$ has the following representation
\begin{eqnarray}  \label{infinite-var}
&& f(s)=s+(1-s)^{\gamma+1}L\left(\frac{1}{1-s}\right), \ \ \ s \in [0,1),
\end{eqnarray}
where $\gamma \in (0,1]$ and $L(.)$ is a function slowly varying at infinity
(s.v.f.). Thus, the process $Z(t), t \geq 0,$ is critical. If $\gamma <1$
the the offspring variance in infinite.

\begin{com}
\label{com1} Let us note that if $\gamma =1$ and $L(t)\rightarrow b$ then
the offspring variance is finite. The results for this case follows directly
from the corresponding results for the multitype Markov processes with
non-homogeneous Poisson immigration studied in \cite{myh}. If $\gamma =1$
the offspring variance can also be infinite, depending on the properties of
the slowly varying function $L(.)$.
\end{com}

The process has an invariant measure whose p.g.f. $U(s)$ is given by
\begin{eqnarray*}
U(s)=\int_0^s \frac{d u}{f(u)-u}, \ \ \ \ 0\le s \le 1.
\end{eqnarray*}
The Kolmogorov backward equation (\ref{kolmo0}) can be written as follows
\begin{eqnarray*}
\int_s^{F(t;s)} \frac{du}{f(u)-u}=\mu t.
\end{eqnarray*}

This leads to $U(F(t;s))=U(s)+\mu t$. Denote by
\begin{equation*}
V(x)=U\left( 1-\frac{1}{x}\right) =\int_{0}^{1-1/x}\frac{du}{f(u)-u},\ \ \ \
x\geq 1,
\end{equation*}%
and let $W(y)$ be the inverse function of $V(x)$. Using the above relations
we get
\begin{eqnarray}
&& \frac{1}{1-F(t;s)} = W\left(\mu t+V\left(\frac{1}{1-s}\right)\right), \ \
s \in [0,1),  \label{q-asymp-2-s}
\end{eqnarray}
Substituting $s=0$ in the above equation we get
\begin{eqnarray}
&& \frac{1}{1-F(t;0)}=W(\mu t).  \label{q-asymp-2-0}
\end{eqnarray}
For $V(x)$ one has
\begin{equation*}
V(x)=\int_{1}^{x}\frac{u^{\gamma -1}}{L(u)}du,x\geq 1.
\end{equation*}%
So $V(x)$ is regularly varying with exponent $\gamma $. Then its inverse $%
W(y)$ is regularly varying with exponent $1/\gamma $, and from (\ref%
{q-asymp-2-0}) we obtain that
\begin{eqnarray}  \label{q-asymp-2}
&& 1-F(t;0) \sim t^{-1/\gamma}L_1(t), \ \ t \to \infty,
\end{eqnarray}
where $L_{1}(t)$ is a slowly varying at infinity function. Let us note that $%
V(x)$ is increasing and $W(y)$ is also increasing (see e.g. \cite{pakes1975}%
, \cite{pakes2010}).

For the p.g.f. of the immigrants we will assume that
\begin{eqnarray}
&& g(s)=1-(1-s)^\alpha l\left(\frac{1}{1-s}\right), \ \ s \in (0,1)
\label{im-fin}
\end{eqnarray}
where $\alpha \in (0,1]$ and $l(x)$ is a function slowly varying at infinity.

\begin{com}
\label{com2} If $\alpha \in (0,1)$ the mean number of immigrants is
infinite. In the case when $\alpha=1$ the mean number of immigrants can be
infinite or finite depending on the s.v.f. $l(.)$. If $\alpha=1$, and $l(x)
\to m \in (0,\infty),$ then $E[I_k]=m$ is finite.
\end{com}

Let us denote
\begin{equation*}
\Psi (x)=\frac{1}{1-g(1-\frac{1}{x})}=\frac{x^{\alpha }}{l(x)},\ x\geq 1.
\end{equation*}%
The function $\Psi (.)$ is non decreasing in $[1,\infty )$. Let us denote by
$\displaystyle\overleftarrow{\Psi }(x),\ x\geq 1,$ its inverse function. It
is also non-decreasing in $[1,\infty )$.

Then $g(s)$ can be written in the following form
\begin{equation*}
g(s)=1-\frac{1}{\Psi \left( \frac{1}{1-s}\right) },\ \ s\in \lbrack 0,1].
\end{equation*}

Further for convenience we will denote (see also (\ref{q-asymp-2-s}))
\begin{eqnarray}
&&q(t;s):=1-g(F(t;s))=(1-F(t;s))^{\alpha }l\left( \frac{1}{1-F(t;s)}\right)
\notag \\
&=&\left[ W\left( \mu t+V\left( \frac{1}{1-s}\right) \right) \right]
^{-\alpha }l\left( W\left( \mu t+V\left( \frac{1}{1-s}\right) \right) \right)
\notag \\
&=&\frac{1}{\Psi \left( W\left( \mu t+V\left( \frac{1}{1-s}\right) \right)
\right) }.  \label{qts}
\end{eqnarray}

\section{Probability for non-visiting the state zero}

\label{sec3}

In this section we derive asymptotic formulas for the probability for
non-visiting zero. Let us denote
\begin{eqnarray}  \label{yanev1}
&&\mathbf{P}\{Y(t)>0\}=1-\Phi(t;0) =1-\exp(-I(t)),
\end{eqnarray}
where
\begin{equation*}
I(t)=\int_{0}^{t}r(t-u)(1-g(F(u;0)))du.
\end{equation*}%
From (\ref{qts}) with $s=0$ and (\ref{q-asymp-2}) it follows that
\begin{eqnarray*}
q(t) &=&q(t;0)=1-g(F(t;0))=\frac{1}{\Psi (W(\mu t))} \\
&=&(1-F(t;0))^{\alpha }l\left( \frac{1}{1-F(t;0)}\right) \sim t^{-\alpha
/\gamma }L_{Q}(t),
\end{eqnarray*}%
where $L_{Q}(t)$ is a s.v.f. at infinity. Let $\displaystyle %
Q(t)=\int_{0}^{t}q(u)du.$

\begin{thm}
\label{thm-2-5} Assume the conditions(\ref{rt-3-g}), (\ref{infinite-var}),
and (\ref{im-fin}) hold.

(i) If $Q(t)\rightarrow \infty ,t\rightarrow \infty ,$ then
\begin{equation}
\mathbf{P}\{Y(t)>0\}\rightarrow 1,\ t\rightarrow \infty .  \label{dt-asimp51}
\end{equation}

(ii) If $\displaystyle Q=\int_{0}^{\infty }q(u)du<\infty $ then
\begin{equation*}
\mathbf{P}\{Y(t)>0\}\rightarrow 1-e^{-\rho Q},\ t\rightarrow \infty .
\end{equation*}
\end{thm}

\begin{proof}
(i) Let $\delta >0$. There exists $T=T(\delta )>0$ such that
for every $t\geq T$, $\rho (1-\delta )\leq r(t)\leq \rho (1+\delta ).$ Then
\begin{equation*}
I(t)=\int_{0}^{t}r(u)q(t-u)du=\int_{0}^{T}+\int_{T}^{t}=I_{1}(t)+I_{2}(t).
\end{equation*}%
For $I_{2}(t)$ we have for $t>T$
\begin{equation*}
I_{2}(t)=\int_{T}^{t}r(u)q(t-u)du\lesseqqgtr \rho (1\pm \delta
)\int_{0}^{t-T}q(u)du\sim \rho (1\pm \delta )Q(t),t\rightarrow \infty .
\end{equation*}%
On the other hand
\begin{equation*}
0\leq I_{1}(t)=\int_{0}^{T}r(u)q(t-u)du\leq
q(t-T).\int_{0}^{T}r(u)du=o(Q(t)),t\rightarrow \infty .
\end{equation*}%
Using the fact that $\delta >0$ was arbitrary we obtain that
\begin{equation*}
I(t)=I_{1}(t)+I_{2}(t)\sim \rho .Q(t)\rightarrow \infty ,t\rightarrow \infty
.
\end{equation*}%
Now from (\ref{yanev1}) we obtain (\ref{dt-asimp51}).

(ii) The proof of this case is similar to the proof of case (i), we
only have to note that in this case $I(t)\rightarrow \rho .Q,\ t\rightarrow \infty $.
\end{proof}

\section{Limit distributions}

\label{limtheor}

We will use the following representation (see (\ref{yanev})),
\begin{eqnarray}  \label{yanev12}
\Phi(t;s)=\exp\left(-I(t;s)\right), \mathrm{\ where \ } I(t;s)=\int_9^t
r(t-u)q(u;s)du.
\end{eqnarray}
Note that we will apply some well-known properties of the regularly varying
and slowly varying functions which can be found in \cite{bgt, Feller}.

\begin{thm}
\label{thm4.6} Assume the conditions (\ref{rt-3-g}), (\ref{infinite-var}),
and (\ref{im-fin}) hold.

(i) If $\displaystyle tq(t)\rightarrow \infty ,t\rightarrow \infty ,$ then
\begin{equation*}
\lim_{t\rightarrow \infty }\mathbf{E}[e^{-\frac{\lambda Y(t)}{\overleftarrow{%
\Psi }(\rho t)}}]=e^{-\lambda ^{\alpha }},
\end{equation*}%
which is the Laplace transform of a one sided stable distribution $%
D_{\alpha}(x)$ and
\begin{equation*}
1-D_{\alpha }(x)\sim x^{-\alpha }/\Gamma (1-\alpha ),\ \ \ x\rightarrow
\infty .
\end{equation*}

(ii) If $\displaystyle tq(t)\rightarrow C\in (0,\infty ),t\rightarrow \infty
,$ then
\begin{equation*}
\lim_{t\rightarrow \infty }\mathbf{E}[e^{-\frac{\lambda Y(t)}{W(t)}%
}]=(1+\lambda ^{\gamma })^{-C\rho },
\end{equation*}%
which is the Laplace transform of a distribution function $G_{\gamma }(x)$
belonging to a normal domain of attraction of a stable law with parameter $%
\gamma $
\begin{equation*}
1-G_{\gamma }(x)\sim x^{-\gamma }Cp/\Gamma (1-\gamma ),\ \ \ x\rightarrow
\infty .
\end{equation*}

(iii) If $\displaystyle Q=\int_{0}^{\infty }q(t)dt<\infty $ then
\begin{equation*}
\displaystyle\lim_{t\rightarrow \infty }\mathbf{E}\left[ s^{Y(t)}|Y(t)>0%
\right] =H(s)=1-\frac{1-\exp \left[ -\rho \Delta (s)\right] }{1-\exp \left[
-\rho Q\right] },
\end{equation*}%
where
\begin{equation*}
\Delta (s)=\int_{0}^{\infty }q(t;s)dt,s\in \lbrack 0,1].
\end{equation*}

(iv) If $\displaystyle \displaystyle tq(t)\rightarrow 0$ but $\displaystyle %
\int_{0}^{\infty}q(t)dt=\infty $ then
\begin{equation*}
\mathbf{P}\left( \frac{A(Y(t))}{A(W(t))}\leq x\right) =x,\mathrm{\ \ for\ }%
x\in (0,1),
\end{equation*}%
where $\displaystyle A(x)=\exp \left( \int_{0}^{V(x)}\frac{du}{\Psi (W(u))}%
\right) $.
\end{thm}

Proof. (i) Note that the condition $\displaystyle tq(t)=t^{1-\alpha /\gamma
}L_{Q}(t)\rightarrow \infty $ is equivalent to the condition $\{(\alpha
<\gamma )\vee (\alpha =\gamma ,L_{Q}(t)\rightarrow \infty )\}.$ Since in
this case $\alpha \leq \gamma $ then by Theorem 3.1 $P(Y(t)>0)\rightarrow
1,t\rightarrow \infty .$ Denote $\displaystyle s(t)=e^{-\lambda /%
\overleftarrow{\Psi }(\rho t)}$. For $\delta \in (0,1)$ we consider
\begin{equation*}
I(t;s(t))=\int_{0}^{t}r(t-u)q(u;s(t))du=\int_{0}^{t\delta }+\int_{t\delta
}^{t}=I_{1}(t;s(t))+I_{2}(t;s(t)).
\end{equation*}%
Since $q(t;s)$ is non increasing in $t\geq 0$ we have
\begin{equation*}
q(t;s(t))R(t(1-\delta ))\leq I_{2}(t;s(t))\leq q(t\delta ;s(t))R(t(1-\delta
)),
\end{equation*}%
where $\displaystyle R(t)=\int_{0}^{t}r(u)du\sim \rho t,t\rightarrow \infty $%
. Further we have that $\displaystyle1-s(t)\sim \frac{\lambda }{%
\overleftarrow{\Psi }(\rho t)},\ t\rightarrow \infty $ and from $%
\displaystyle q(t)\sim \frac{1}{\Psi (W(\mu t))}$ one has $\displaystyle\mu
t\sim V\left( \overleftarrow{\Psi }\left( \frac{1}{q(t)}\right) \right) $.
Then as $t\rightarrow \infty $,
\begin{equation*}
\frac{\mu t}{V\left( \frac{1}{1-s(t)}\right) }\sim \frac{V\left(
\overleftarrow{\Psi }\left( \frac{1}{q(t)}\right) \right) }{V\left( \frac{%
\overleftarrow{\Psi }(\rho t)}{\lambda }\right) }\sim \lambda ^{-\gamma }%
\frac{V\left( \overleftarrow{\Psi }\left( \frac{1}{q(t)}\right) \right) }{%
V\left( \overleftarrow{\Psi }(\rho t)\right) }\sim \lambda ^{-\gamma }\frac{1%
}{[q(t)\rho t]^{\gamma /\alpha }}\rightarrow 0.
\end{equation*}%
From this relation, using the uniform convergence of regularly varying
functions we get that as $t\rightarrow \infty $,
\begin{eqnarray*}
&&q(t;s(t))R(t(1-\delta ))=\frac{R(t(1-\delta ))}{\Psi \left( W\left(
V\left( \frac{1}{1-s(t)}\right) \left( \frac{\mu t}{V(\frac{1}{1-s(t)})}%
+1\right) \right) \right) } \\
&\sim &\frac{\rho t(1-\delta )}{\Psi \left( W\left( V\left( \frac{1}{1-s(t)}%
\right) \right) \right) }\sim \frac{\rho t(1-\delta )}{\Psi (\frac{1}{1-s(t)}%
)}\sim \frac{\rho t(1-\delta )}{\Psi \left( \frac{\overleftarrow{\Psi }(\rho
t)}{\lambda }\right) } \\
&\sim &\lambda ^{\alpha }\frac{\rho t(1-\delta )}{\Psi \left( \overleftarrow{%
\Psi }(\rho t)\right) }\sim \lambda ^{\alpha }\rho t(1-\delta ).\frac{1}{%
\rho t}\rightarrow (1-\delta )\lambda ^{\alpha }.
\end{eqnarray*}%
In the same way one has that
\begin{equation*}
q(t\delta ;s(t))\rho t(1-\delta )\rightarrow \lambda ^{\alpha }(1-\delta ),\
t\rightarrow \infty .
\end{equation*}%
Notice that for every $\varepsilon >0$ and large enough $t$
\begin{equation*}
(\rho -\varepsilon )\int_{0}^{t\delta }q(u;s(t))du\leq I_{1}(t;s(t))\leq
(\rho +\varepsilon )\int_{0}^{t\delta }q(u;s(t))du
\end{equation*}%
Having in mind that $W(.)$ is increasing and $V(x)>0$ for any $x\geq 1$, we
get
\begin{eqnarray*}
\int_{0}^{t\delta }q(u;s(t))du &=&\int_{0}^{t\delta }\frac{du}{\Psi \left(
W\left( \mu u+V\left( \frac{1}{1-s(t)}\right) \right) \right) } \\
&\leq &\int_{0}^{t\delta }\frac{du}{\Psi \left( W\left( V\left( \frac{1}{%
1-s(t)}\right) \right) \right) } \\
&\sim &\frac{t\delta \lambda ^{\alpha }}{\Psi \left( \overleftarrow{\Psi }%
(\rho t)\right) }\sim \delta \lambda ^{\alpha }t\frac{1}{\rho t}=\delta
\lambda ^{\alpha }/\rho .
\end{eqnarray*}%
Therefore
\begin{equation*}
(1-\delta )\lambda ^{\alpha }\leq \liminf_{t\rightarrow \infty
}I(t;s(t))\leq \limsup_{t\rightarrow \infty }I(t;s(t))\leq (1-\delta
)\lambda ^{\alpha }+\delta \lambda ^{\alpha }(\rho +\varepsilon )/\rho .
\end{equation*}%
Since $\delta $ was arbitrary then as $\delta \rightarrow 0$, $%
\lim_{t\rightarrow \infty }I(t;s(t))=\lambda ^{\alpha },$ which together
with (\ref{yanev12}) completes the proof of this case.

(ii) Note that in this case $\frac{\alpha }{\gamma }=1$. Denote by $%
s(t)=e^{-\lambda /W(\mu t)}$ and choose $\varepsilon \in (0,1)$ fixed. Then
\begin{equation*}
1-s(t)\in \left( (1-\varepsilon )\frac{\lambda }{W(\mu t)};(1+\varepsilon )%
\frac{\lambda }{W(\mu t)}\right) ,
\end{equation*}%
for every $t$ large enough and there exists $T>0$ such that
\begin{equation*}
\rho (1-\varepsilon )\leq r(t)\leq \rho (1+\varepsilon ),\ \ \ \mathrm{\
for\ every\ }t\geq T.
\end{equation*}%
Since $q(t;s)$ is non increasing for $s\in \lbrack 0,1]$ we have for every $%
t $ large enough and $u\in \lbrack 0,t]$ that
\begin{equation*}
q\left( u;1-\frac{\lambda (1+\varepsilon )}{W(\mu t)}\right) \leq q\left(
u;s(t)\right) \leq q\left( u;1-\frac{\lambda (1-\varepsilon )}{W(\mu t)}%
\right) .
\end{equation*}%
Let
\begin{equation*}
I(t;s(t))=\int_{0}^{t}r(t-u)q(u;s(t))du=\int_{0}^{t-T}+%
\int_{t-T}^{t}=I_{1}(t;s(t))+I_{2}(t;s(t)).
\end{equation*}%
Then for $I_{1}(t;s(t))$ we get
\begin{eqnarray*}
&&\int_{0}^{t-T}r(t-u)q\left( u;1-\frac{\lambda (1+\varepsilon )}{W(\mu t)}%
\right) du\leq I_{1}(t;s(t)) \\
&\leq &\int_{0}^{t-T}r(t-u)q\left( u;1-\frac{\lambda (1-\varepsilon )}{W(\mu
t)}\right) du.
\end{eqnarray*}%
Using (\ref{qts}) and that $V(.)$ and $W(.)$ are inverse to each other, $%
V(.) $ varies regularly with exponent $\gamma $, $\Psi (.)$ varies regularly
with exponent $\alpha $, and $\alpha =\gamma ,$ one gets
\begin{eqnarray*}
&&q\left( u;1-\frac{\lambda (1\pm \varepsilon )}{W(\mu t)}\right) =\frac{1}{%
\Psi \left( W\left( \mu u+V\left( \frac{W(\mu t)}{\lambda (1\pm \varepsilon )%
}\right) \right) \right) } \\
&\sim &\frac{1}{\Psi \left( W\left( \mu u+\frac{\mu t}{(\lambda (1\pm
\varepsilon ))^{\gamma }}\right) \right) }\sim \frac{1}{\Psi \left( W(\mu
t)\left( \frac{u}{t}+\frac{1}{(\lambda (1\pm \varepsilon ))^{\gamma }}%
\right) ^{\frac{1}{\gamma }}\right) } \\
&\sim &\frac{1}{\Psi \left( W(\mu t)\right) \left( \frac{u}{t}+\frac{1}{%
(\lambda (1\pm \varepsilon ))^{\gamma }}\right) }=q(t)\left( \frac{u}{t}+%
\frac{1}{(\lambda (1\pm \varepsilon ))^{\gamma }}\right) ^{-1}.
\end{eqnarray*}%
Here we use the uniform convergence of slowly varying functions. Now for
every $t$ large enough it follows that
\begin{eqnarray}
&&\rho (1-\varepsilon )tq(t)\frac{1}{t}\int_{0}^{t-T}\left( \frac{u}{t}+%
\frac{1}{(\lambda (1-\varepsilon ))^{\gamma }}\right) ^{-1}du  \notag \\
&\leq &I_{1}(t;s(t))\leq \rho (1+\varepsilon )tq(t)\frac{1}{t}%
\int_{0}^{t-T}\left( \frac{u}{t}+\frac{1}{(\lambda (1+\varepsilon ))^{\gamma
}}\right) ^{-1}du.  \label{Intepsilon}
\end{eqnarray}%
Substituting $v=\frac{u}{t}$ and letting $t\rightarrow \infty $ one gets
\begin{eqnarray*}
&&\frac{1}{t}\int_{0}^{t-T}\left( \frac{u}{t}+\frac{1}{(\lambda (1\pm
\varepsilon ))^{\gamma }}\right) ^{-1}du \\
&=&\int_{0}^{1-T/t}\left( v+\frac{1}{(\lambda (1\pm \varepsilon ))^{\gamma }}%
\right) ^{-1}dv\rightarrow \int_{0}^{1}\left( v+\frac{1}{(\lambda (1\pm
\varepsilon ))^{\gamma }}\right) ^{-1}dv.
\end{eqnarray*}%
Having in mind that $tq(t)\rightarrow C\in (0,\infty )$ it follows that
\begin{eqnarray*}
&&\rho (1-\varepsilon )C\int_{0}^{1}\left( \frac{u}{t}+\frac{1}{(\lambda
(1-\varepsilon ))^{\gamma }}\right) ^{-1} \\
&\leq &\liminf_{t\rightarrow \infty }I_{1}(t;s(t))\leq \limsup_{t\rightarrow
\infty }I_{1}(t;s(t)) \\
&\leq &\rho (1+\varepsilon )C\int_{0}^{1}\left( \frac{u}{t}+\frac{1}{%
(\lambda (1+\varepsilon ))^{\gamma }}\right) ^{-1}.
\end{eqnarray*}%
Since $\varepsilon >0$ was arbitrary then we get
\begin{equation*}
\lim_{t\rightarrow \infty }I_{1}(t;s(t))=\rho C\int_{0}^{1}\left( v+\frac{1}{%
\lambda ^{\gamma }}\right) ^{-1}dv=\log (1+\lambda ^{\gamma })^{\rho C}.
\end{equation*}%
For $I_{2}(t;s(t))$ one gets
\begin{equation*}
I_{2}(t;s(t))=\int_{t-T}^{t}r(t-u)q(u;s(t))du\leq
q(t)\int_{0}^{T}r(u)du\rightarrow 0,t\rightarrow \infty .
\end{equation*}%
Therefore
\begin{equation*}
\lim_{t\rightarrow \infty }I(t;s(t))=\log (1+\lambda ^{\gamma })^{\rho C}
\end{equation*}%
which together with (\ref{yanev12}) completes the convergence to the Laplace
transform $\varphi (\lambda )=(1+\lambda ^{\gamma })^{-C\rho }$. \ Since

\begin{equation*}
\int_{0}^{\infty }e^{-\lambda x}(1-G_{\gamma }(x))dx=\lambda ^{-1}(1-\varphi
(\lambda ))\sim \rho C\lambda ^{\gamma -1},\lambda \rightarrow 0,
\end{equation*}%
\linebreak then by the Tauberian theorem (see \cite{Feller}, Ch. XIII,
Theorem 5.4) one obtains the statement of this case.

(iii) Since $q(t;s)\leq q(t)$ then $\displaystyle \Delta(s)=\int_{0}^{\infty
}q(u;s)du<\infty $ for $s\in \lbrack 0,1]$. Let $\delta \in (0,1)$ be fixed
and
\begin{equation*}
I(t;s)=\int_{0}^{t}r(t-u)q(u;s)du=\int_{0}^{t\delta }+\int_{t\delta
}^{t}=I_{1}(t;s)+I_{2}(t;s).
\end{equation*}%
Let $\varepsilon \in (0,1)$ be fixed. Then for $t$ large enough one has
\begin{equation*}
\rho (1-\varepsilon )\int_{0}^{t\delta }q(u;s)du\leq I_{1}(t;s)\leq \rho
(1+\varepsilon )\int_{0}^{t\delta }q(u;s)du.
\end{equation*}%
Therefore
\begin{equation*}
\rho (1-\varepsilon )\Delta (s)\leq \liminf_{t\rightarrow \infty
}I_{1}(t;s)\leq \limsup_{t\rightarrow \infty }I_{1}(t;s)\leq \rho
(1+\varepsilon )\Delta (s).
\end{equation*}%
On the other hand for $I_{2}(t;s)$ we obtain
\begin{eqnarray*}
&&0\leq I_{2}(t;s)=\int_{t\delta }^{t}r(t-u)q(u;s)du \\
&\leq &q(t)\int_{0}^{t(1-\delta )}r(u)du\sim q(t)t(1-\delta )\sup_{0\leq
x\leq t(1-\delta )}r(x)\rightarrow 0,\ t\rightarrow \infty .
\end{eqnarray*}%
Since $\varepsilon $ was arbitrary then we conclude that $\lim_{t\rightarrow
\infty }I(t;s)=\rho \Delta (s),$ which is equivalent to $\lim_{t\rightarrow
\infty }\Phi (t,s)=\exp (-r\Delta (s))$ (see (\ref{yanev12})). Then from
\begin{equation*}
\mathbf{E}\left[ s^{Y(t)}|Y(t)>0\right] =1-\frac{1-\Phi (t,s)}{1-\Phi (t,0)}
\end{equation*}%
with Theorem \ref{thm-2-5} (ii) \ one obtains the result.

(iv) Introduce
\begin{equation*}
B(x)=\exp \left( \int_{0}^{x}\frac{du}{\Psi (W(u))}\right) =\exp \left(
\int_{0}^{x}\frac{uq(u/\mu )}{u}du\right) .
\end{equation*}%
Since $tq(t)\rightarrow 0,t\rightarrow \infty ,$ then $B(x)$ is slowly
varying at infinity and $B(x)$ is nondecreasing. Denote by $\overleftarrow{B}%
(x)$ the inverse function of $B(x)$. Clearly $A(x)=B(V(x)).$ Let us consider
the integral (see (\ref{yanev}) and (\ref{qts}))
\begin{equation*}
I(t;s(t))=\int_{0}^{t}r(t-u)q(u;s(t))du,
\end{equation*}%
where $s(t)=\exp (-\lambda /W(\overleftarrow{B}(\sigma B(t))))$ for $%
0<\sigma <1,\ \lambda >0,t\geq 0.$ Let $\varepsilon \in (0,1)$ be fixed.
There exists $T>0$ such that for every $t\geq T$, one has $\rho
(1-\varepsilon )\leq r(t)\leq \rho (1+\varepsilon ).$ Then
\begin{equation*}
I(t;s(t))=\int_{0}^{t-T}+\int_{t-T}^{t}=I_{1}(t;s(t))+I_{2}(t;s(t)).
\end{equation*}%
For every $t\geq T$ one has
\begin{equation*}
\rho (1-\varepsilon )J(t;s(t))\leq I_{1}(t;s(t))\leq \rho (1+\varepsilon
)J(t;s(t)),
\end{equation*}%
where (see also (\ref{qts}))
\begin{eqnarray*}
&&J(t;s(t)):=\int_{0}^{t-T}q(u;s(t))du=\int_{0}^{t-T}\frac{du}{\Psi \left[
W\left( \mu u+V\left( \frac{1}{1-s(t)}\right) \right) \right] } \\
&=&\left. \log B\left( \mu u+V\left( \frac{1}{1-s(t)}\right) \right)
\right\vert _{0}^{t-T}=\log \frac{B\left( \mu (t-T)+V\left( \frac{1}{1-s(t)}%
\right) \right) }{B\left( V\left( \frac{1}{1-s(t)}\right) \right) }.
\end{eqnarray*}%
For $1-s(t)$ one has that
\begin{equation*}
1-s(t)=1-\exp \left( -\frac{\lambda }{W(\overleftarrow{B}(\sigma B(t)))}%
\right) \sim \frac{\lambda }{W(\overleftarrow{B}(\sigma B(t)))},t\rightarrow
\infty ,
\end{equation*}%
because $W(\overleftarrow{B}(\sigma B(t)))\rightarrow \infty $ for every
fixed $\sigma \in (0,1)$. Recall that $V(.)$ and $W(.)$ are inverse to each
other and $V(.)$ varies regularly with exponent $\gamma \in (0,1]$, we get
\begin{equation*}
V\left( \frac{1}{1-s(t)}\right) \sim V\left( \frac{W(\overleftarrow{B}%
(\sigma B(t)))}{\lambda }\right) \sim \frac{\overleftarrow{B}(\sigma B(t))}{%
\lambda ^{\gamma }},\ t\rightarrow \infty .
\end{equation*}%
Therefore
\begin{equation*}
B\left( V\left( \frac{1}{1-s(t)}\right) \right) \sim B\left( \frac{%
\overleftarrow{B}(\sigma B(t))}{\lambda ^{\gamma }}\right) \sim \sigma
B(t),t\rightarrow \infty ,
\end{equation*}%
because $B(x)$ is a slowly varying function at infinity. Further we have
\begin{equation*}
B\left( \mu (t-T)+V\left( \frac{1}{1-s(t)}\right) \right) \sim B\left( t-T+%
\frac{\overleftarrow{B}(\sigma B(t))}{\mu \lambda ^{\gamma }}\right) .
\end{equation*}%
By the mean value theorem one has
\begin{equation*}
B\left( t-T+\frac{\overleftarrow{B}(\sigma B(t))}{\mu \lambda ^{\gamma }}%
\right) -B(t)=B^{\prime }(\xi _{t})\left[ \frac{\overleftarrow{B}(\sigma
B(t))}{\mu \lambda ^{\gamma }}-T\right] ,
\end{equation*}%
where $\displaystyle\xi _{t}\in \left( t,t+\frac{\overleftarrow{B}(\sigma
B(t))}{\mu \lambda ^{\gamma }}-T\right) .$ Having in mind that $%
\displaystyle
B^{\prime }(x)=\frac{B(x)}{\Psi (W(x))}$ we get
\begin{eqnarray*}
&&B\left( t-T+\frac{\overleftarrow{B}(\sigma B(t))}{\mu \lambda ^{\gamma }}%
\right) -B(t)=\frac{B(\xi _{t})}{\Psi (W(\xi _{t}))}\left[ \frac{%
\overleftarrow{B}(\sigma B(t))}{\mu \lambda ^{\gamma }}-T\right] \\
&\leq &\frac{B(t+\frac{\overleftarrow{B}(B(t))}{\mu \lambda ^{\gamma }})}{%
\Psi (W(t))}\frac{\overleftarrow{B}(B(t))}{\mu \lambda ^{\gamma }}=\frac{B(t+%
\frac{t}{\mu \lambda ^{\gamma }})}{\Psi (W(t))}\frac{\overleftarrow{B}(B(t))%
}{\mu \lambda ^{\gamma }} \\
&=&B(t(1+\frac{1}{\mu \lambda ^{\gamma }}))\frac{tq(t/\mu )}{\mu \lambda
^{\gamma }},
\end{eqnarray*}%
because $\displaystyle\Psi (W(x))=\frac{1}{q(x/\mu )}$, $B(.)$, and resp. $%
\overleftarrow{B}(.)$ are nondecreasing. Therefore,
\begin{equation*}
B(t)\leq B\left( t-T+\frac{\overleftarrow{B}(\sigma B(t))}{\mu \lambda
^{\gamma }}\right) \leq B(t(1+\frac{1}{\mu \lambda ^{\gamma }}))\left( 1+%
\frac{tq(t/\mu )}{\mu \lambda ^{\gamma }}\right) ,
\end{equation*}%
and for large enough $t$ we have
\begin{equation*}
\log \frac{B(t)}{\sigma B(t)}\leq J(t;s(t))\leq \log \frac{B(t(1+\frac{1}{%
\mu \lambda ^{\gamma }}))(1+\frac{tq(t/\mu )}{\mu \lambda ^{\gamma }})}{%
\sigma B(t)}
\end{equation*}%
or
\begin{equation*}
\log \frac{1}{\sigma }\leq J(t;s(t))\leq \log \frac{1}{\sigma }+\log \frac{%
B(t(1+\frac{1}{\mu \lambda ^{\gamma }}))(1+\frac{tq(t/\mu )}{\mu \lambda
^{\gamma }})}{B(t)}.
\end{equation*}%
Having in mind that $B(t)$ is a s.v.f. and $tq(t)\rightarrow 0,t\rightarrow
\infty ,$ we obtain that
\begin{equation*}
\lim_{t\rightarrow \infty }J(t;s(t))=-\log \sigma ,
\end{equation*}%
and then
\begin{equation*}
\rho (1-\varepsilon )(-\log \sigma )\leq \liminf_{t\rightarrow \infty
}I_{1}(t;s(t))\leq \limsup_{t\rightarrow \infty }I_{1}(t;s(t))\leq \rho
(1+\varepsilon )(-\log \sigma ).
\end{equation*}%
For $I_{2}(t;s(t))$ we get
\begin{equation*}
I_{2}(t;s(t))=\int_{t-T}^{t}r(t-u)q(u;s(t))du\leq
q(t-T)\int_{0}^{T}r(u)du\rightarrow 0,t\rightarrow \infty .
\end{equation*}%
Since $\varepsilon $ was arbitrary it follows that
\begin{equation*}
\lim_{t\rightarrow \infty }I(t;s(t))=-\log \sigma ,\ \ \ \sigma \in (0,1).
\end{equation*}%
Therefore (see (\ref{yanev12}))
\begin{equation*}
\lim_{t\rightarrow \infty }\Phi (t;s(t))=\sigma ,\ \ \ \sigma \in (0,1),
\end{equation*}%
which implies that for every $x>0$
\begin{equation*}
\lim_{t\rightarrow \infty }\mathbf{P}\left( \frac{Y(t)}{W(\overleftarrow{B}%
(\sigma B(t)))}\leq x\right) =\sigma ,
\end{equation*}%
that is
\begin{equation*}
\lim_{t\rightarrow \infty }\mathbf{P}\left( \frac{Y(t)}{W(\overleftarrow{B}%
(\sigma B(t)))}\leq 1\right) =\sigma .
\end{equation*}%
By the following chain of equalities
\begin{eqnarray*}
&&\mathbf{P}\left( \frac{Y(t)}{W(\overleftarrow{B}(\sigma B(t)))}\leq
1\right) =\mathbf{P}\left( Y(t)\leq W(\overleftarrow{B}(\sigma B(t)))\right)
\\
&=&\mathbf{P}\left( B(V(Y(t)))\leq B(V(W(\overleftarrow{B}(\sigma
B(t)))))\right) \\
&=&\mathbf{P}\left( B(V(Y(t)))\leq \sigma B(t)\right) =\mathbf{P}\left(
\frac{B(V(Y(t)))}{B(t)}\leq \sigma \right) \\
&=&\mathbf{P}\left( \frac{B(V(Y(t)))}{B(V(W(t)))}\leq \sigma \right) =%
\mathbf{P}\left( \frac{A(Y(t))}{A(W(t))}\leq \sigma \right) ,
\end{eqnarray*}%
we complete the proof.

\end{document}